\documentclass[11pt]{amsart}

\usepackage{amsrefs,amstext,amsmath,amsthm,amsfonts,amssymb,enumerate,cite}

\newtheorem{theorem}{Theorem}[section]
\newtheorem{proposition}[theorem]{Proposition}
\newtheorem{lemma}[theorem]{Lemma}
\newtheorem{corollary}[theorem]{Corollary}
\newtheorem{definition}[theorem]{Definition}
\newtheorem{remark}[theorem]{Remark}
\newtheorem{remarks}[theorem]{Remarks}

\newcommand{\CC}{\ensuremath{\mathbb{C}}} 
\newcommand{\NN}{\ensuremath{\mathbb{N}}} 
\newcommand{\ZZ}{\ensuremath{\mathbb{Z}}} 
\newcommand{\lspan}{\mathop{\mathrm{span}}} 
\newcommand{\nrm}[1][\cdot]{\Vert #1 \Vert} 
\newcommand{\bignrm}[1][\cdot]{\left\Vert #1 \right\Vert} 
\newcommand{\ran}{\mathop{\mathrm{ran}}} 
\newcommand{\supp}{\mathop{\mathrm{supp}}} 
\newcommand{\integ}[3][]{\ensuremath{\int_{#1}#2\,d#3}} 
\newcommand{\id}{\mathrm{id}} 
\newcommand\esssup{\mathop{\mathrm{esssup}}} 
\newcommand{\diag}{\mathrm{diag}} 
\newcommand{\dualp}[2]{\left\langle #1 , #2 \right\rangle} 

\newcommand{\HH}{\ensuremath{\mathcal{H}}} 
\newcommand{\ip}[2]{\left\langle #1 , #2 \right\rangle} 

\newcommand{\inv}{^{-1}} 
\newcommand{\HS}{\mathrm{HS}} 
\newcommand{\SL}{\mathrm{SL}} 
\newcommand{\Aut}{\mathop{\mathrm{Aut}}} 

\newcommand{\CA}{$C^*$-algebra}
\newcommand{\vNA}{von Neumann algebra}
\newcommand{\Bd}[1][\HH]{\ensuremath{\mathcal{B}(#1)}} 
\newcommand{\unA}{1_A} 

\newcommand{\vN}[1][G]{\ensuremath{\mathrm{vN}(#1)}} 
\newcommand{\rgCst}[1][G]{\ensuremath{C^*_r(#1)}} 
\newcommand{\cros}[3]{\ensuremath{#1 \rtimes_{#3} \! #2}} 
\newcommand{\rcros}[3]{\ensuremath{#1 \rtimes_{{#3},r}  #2}} 
\newcommand{\repcros}[4]{\ensuremath{#1 \rtimes_{{#3},#4}  #2}} 
\newcommand{\thcrs}{\repcros{A}{G}{\alpha}{\theta}} 
\newcommand{\crs}{\cros{A}{G}{\alpha}} 
\newcommand{\rcrs}{\rcros{A}{G}{\alpha}} 
\newcommand{\vNcros}[3]{\ensuremath{#1 \rtimes_{#3}^{\mathrm{vN}} \! #2}} 
\newcommand{\wstrepcros}[4]{\ensuremath{#1 \rtimes_{{#3},#4}^{\mathrm{w}^*}  #2}} 
\newcommand{\wstthcrs}{\wstrepcros{A}{G}{\alpha}{\theta}} 
\newcommand{\EE}{\ensuremath{\mathcal{E}}} 

\newcommand{\CB}[1][A]{\ensuremath{\mathcal{CB}(#1)}} 
\newcommand{\CBw}[1][A]{\ensuremath{\mathcal{CB}_\sigma(#1)}} 
\newcommand{\cb}{{\mathrm{cb}}} 

\newcommand{\falg}[1][G]{\ensuremath{A(#1)}} 
\newcommand{\cfalg}[4][]{\ensuremath{\mathcal{A}^{#1}(#2,#3,#4)}} 
\newcommand{\cfal}{\cfalg[\theta]{A}{G}{\alpha}} 
\newcommand{\Aa}{\mathcal{A}} 

\newcommand{\Mult}{\mathrm{M}} 
\newcommand{\Mcb}{{\mathrm{M}^{\mathrm{cb}}}} 
\newcommand{\HSmults}[3]{\ensuremath{\mathfrak{S}(#1,#2,#3)}} 
\newcommand{\HSm}{\HSmults{A}{G}{\alpha}} 

\title{Weak amenability for dynamical systems}
\author[A. McKee]{Andrew McKee}
\address{Department of Mathematical Sciences, Chalmers University of Technology and the University of Gothenburg, Gothenburg SE-412 96, Sweden}
\email{amckee240@qub.ac.uk}

\begin{document}

\bibliographystyle{plain}

\begin{abstract}
Using the recently developed notion of a Herz--Schur multiplier of a $C^*$-dynamical system we introduce weak amenability of $C^*$- and $W^*$-dynamical systems. As a special case we recover Haagerup's characterisation of weak amenability of a discrete group. We also consider a generalisation of the Fourier algebra to crossed products and study its multipliers.
\end{abstract}



\maketitle

\section{Introduction}
\label{sec:intro}

\noindent
Among the many characterisations of amenability of a locally compact group $G$ is Leptin's Theorem~\cite{Lep68}: $G$ is amenable if and only if the Fourier algebra of $G$ has a bounded approximate identity.
The idea to weaken the latter condition, by requiring the approximate identity to be bounded in a different norm, goes back to Haagerup~\cite{Haaun16}.
Following this, Cowling--Haagerup~\cite{CH89} formally defined weak amenability, explored some equivalent conditions, and introduced the Cowling--Haagerup (or weak amenability) constant. This constant has been computed for a large number of groups --- see Brown--Ozawa~\cite[Theorem 12.3.8]{BrO08} and the references given by Knudby~\cite{Knth14}.
An overview of the literature surrounding weak amenability can be found in the thesis of Knudby~\cite[Section 5]{Knth14}.

Weak amenability is an example of a property defined in terms of functions on a group which can be characterised by an approximation property of the group \vNA\ and/or group \CA\ (see \cite[Chapter 12]{BrO08} for several examples of such properties); the aim of this paper is to extend this idea to crossed products.
A \CA\ $A$ is said to have the \emph{completely bounded approximation property} (CBAP) if there exists a net $(T_\gamma)$ of finite rank completely bounded maps on $A$ such that $T_\gamma \to \id_A$ in the point-norm topology and $\sup_\gamma \nrm[T_\gamma]_\text{cb} = C < \infty$. The infimum of all such constants $C$ is denoted $\Lambda_\text{cb}(A)$.
Similarly, a \vNA\ $M$ is said to have the \emph{weak* completely bounded approximation property} (weak* CBAP) if there exists a net $(R_\gamma)$ of ultraweakly continuous, finite rank, completely bounded maps on $M$ such that $R_\gamma \to \id_M$ in the point-weak* topology and $\sup_\gamma \nrm[R_\gamma]_\text{cb} = C < \infty$; again, the infimum of all such constants $C$ is denoted $\Lambda_\text{cb}(M)$.
A locally compact group $G$ is called \emph{weakly amenable} if there exists a net of compactly supported Herz--Schur multipliers on $G$, uniformly bounded in the Herz--Schur multiplier norm, converging uniformly to 1 on compact sets. Haagerup~\cite[Theorem 2.6]{Haaun16} proved that a discrete group is weakly amenable if and only if the reduced group \CA\ has the completely bounded approximation property, if and only if the group \vNA\ has the weak* completely bounded approximation property.

In this paper we define weak amenability of $C^*$- and $W^*$-dynamical systems and characterise a weakly amenable system in terms of the completely bounded approximation property of the corresponding crossed product. The results in this direction, Theorems~\ref{th:systemweaklyamenableiffnormclosedcrossedprodhasCBAP} and \ref{th:characterisationweakamenabilityandCBAPcrossedproducts}, may be seen as a generalisation of Haagerup's result above.
Haagerup--Kraus~\cite[Section 3]{HKr94} have studied $W^*$-dynamical systems under the assumption that $G$ is weakly amenable; Proposition~\ref{pr:weakamenabilitycovarianceequivalence} was motivated by their Theorem~3.2(b) and Remark~3.10.

In Section~\ref{sec:prelims} we review the definitions and results surrounding the notion of a Herz--Schur multiplier of a $C^*$-dynamical system. 
Section~\ref{sec:fourieralgebras} is motivated by the description of Herz--Schur multipliers as completely bounded multipliers of the Fourier algebra; we view the predual of (the enveloping von~Neumann algebra of) the reduced crossed product as consisting of vector-valued functions on the group, and describe the completely bounded multipliers of this space as certain Herz--Schur multipliers of the associated dynamical system. 
In Section~\ref{sec:weakamenability} we define weak amenability of $C^*$- and $W^*$-dynamical systems, and characterise in terms of the completely bounded approximation property of the associated crossed product.

\section{Preliminaries}
\label{sec:prelims}

\noindent
In this section we review the definitions and results of \cite{MTT18} required later, as well as establishing notation.
Throughout, $G$ will denote a second-countable, locally compact, topological group, endowed with left Haar measure $m$; integration on $G$, with respect to $m$, over the variable $s$, is simply denoted $ds$. 
Write $\lambda^G$ for the left regular representation of $G$ on $L^2(G)$, and the corresponding representation of $L^1(G)$.
The reduced group \CA\ \rgCst\ and group \vNA\ \vN\ of $G$ are, respectively, the closure of $\lambda^G(L^1(G))$ in the norm and weak* topology of $\Bd[L^2(G)]$; we also have $\vN = \{ \lambda^G_s : s \in G \}''$.

Let $A$ be a unital, separable, \CA, which unless otherwise stated will be considered as a $C^*$-subalgebra of \Bd[\HH_A], where $\HH_A$ denotes the Hilbert space of the universal representation of $A$. Let $\alpha : G \to \Aut(A)$ be a group homomorphism which is continuous in the point-norm topology, {\it i.e.}\ for all $a \in A$ the map $s \mapsto \alpha_s(a)$ is continuous from $G$ to $A$; in short, consider a $C^*$-dynamical system $(A,G,\alpha)$.
Let $\theta$ be a faithful representation of $A$ on $\HH_\theta$ and define representations of $A$ and $G$ on $L^2(G,\HH_\theta)$ by
\begin{equation*}\label{eq:picovariantpair}
    \big( \pi^\theta(a)\xi \big)(s) := \theta(\alpha_{s\inv})(a) \big(\xi(s) \big), \quad (\lambda^\theta_t \xi)(s) := \xi(t\inv s) ,
\end{equation*}
for all $a\in A,\ s,t\in G,\ \xi \in L^2(G,\HH_\theta)$. 
It is easy to check that
\[
    \pi^\theta \big(\alpha_t(a) \big) = \lambda^\theta_t \pi^\theta(a) (\lambda^\theta_t)^* ,\quad a \in A,\ t \in G .
\]
The pair $(\pi^\theta,\lambda^\theta)$ is therefore a \emph{covariant representation} of $(A,G,\alpha)$. 
Thus we obtain a representation $\pi^\theta \rtimes \lambda^\theta$ of the Banach $*$-algebra $L^1(G,A)$ on $\Bd[L^2(G,\HH_\theta)]$ given by
\[
    \pi^\theta \rtimes \lambda^\theta (f) := \integ[G]{\pi^\theta \big(f(s) \big)\lambda^\theta_s}{s} ,\quad f \in L^1(G,A) .
\]
The \emph{reduced crossed product} of $A$ by $G$ is defined as the closure of $(\pi^\theta \rtimes \lambda^\theta)(L^1(G,A))$ in the operator norm of $\Bd[L^2(G,\HH_\theta)]$; it does not depend on the choice of faithful representation $\theta$ so we will often omit the superscript $\theta$ from our notation, and denote the reduced crossed product by \rcrs, writing $\thcrs$ when we wish to emphasise the choice of $\theta$. 
The \emph{full crossed product} of $A$ by $G$, denoted $\crs$, is the \CA\ obtained by completing $L^1(G,A)$ in the universal norm 
\[
    \nrm[f] := \sup\{ \nrm[\rho \rtimes \tau(f)] : \text{ $(\rho , \tau)$ is a covariant representation of $(A,G,\alpha)$ } \} .
\]
We refer to Pedersen~\cite[Chapter 7]{Ped79} and Williams~\cite{Wil07} for the details of these constructions.

In \cite{MTT18} the present author, with Todorov and Turowska, introduced and studied Herz--Schur multipliers of a $C^*$-dynamical system, extending the classical notion of a Herz--Schur multiplier (see de~Canni\`ere--Haagerup~\cite{dCH85}). We now recall the definitions and results needed here; the classical definitions  of Herz--Schur multipliers are the special case $A=\CC$ of the definitions below.
A bounded function $F : G \to \Bd[A]$ will be called \emph{pointwise-measurable} if, for every $a \in A$, the map $s \mapsto F(s)(a)$ is a weakly-measurable function from $G$ to $A$.
For each $f \in L^1(G,A)$ define $F \cdot f (s) := F(s)(f(s))$ ($s\in G$). If $F$ is bounded and pointwise-measurable then $F \cdot f$ is weakly measurable and $\nrm[F \cdot f]_1 \leq \sup_{s \in G}\nrm[F(s)] \nrm[f]_1$, so $F \cdot f \in L^1(G,A)$ for every $f \in L^1(G,A)$.

\begin{definition}\label{de:HSmultiplier}
A bounded, pointwise-measurable, function $ F : G \to \CB$ will be called a \emph{Herz--Schur ($A,G,\alpha$)-multiplier} if the map
\[
    S_F : (\pi \rtimes \lambda)(L^1(G,A)) \to (\pi \rtimes \lambda)(L^1(G,A)) ;\ S_F \big( (\pi \rtimes \lambda)(f) \big) := (\pi \rtimes \lambda)(F \cdot f)
\]
is completely bounded; if this is the case then $S_F$ has a unique extension to a completely bounded map on \rcrs. The set of all Herz--Schur $(A,G,\alpha)$-multipliers is an algebra with respect to the obvious operations; we denote it by \HSm\ and endow it with the norm $\nrm[F]_\HS := \nrm[S_F]_\cb$.
\end{definition}

Since the closure of $(\pi^\theta \rtimes \lambda^\theta)(L^1(G,A))$ is isomorphic to \rcrs\ (see {\it e.g.}\ \cite[Theorem 7.7.5]{Ped79}) it follows that $F$ is a Herz--Schur $(A,G,\alpha)$-multiplier if and only if the map
\[
    S_F^\theta : (\pi^\theta \rtimes \lambda^\theta)(f) \mapsto (\pi^\theta \rtimes \lambda^\theta)(F \cdot f) , \quad f \in L^1(G,A),
\]
is completely bounded, so Herz--Schur $(A,G,\alpha)$-multipliers can be defined using any faithful representation of $A$ \cite[Remark 3.2(ii)]{MTT18}.
Let $\alpha^\theta : G \to \Aut(\theta(A))$ be given by $\alpha^\theta_t(\theta(a)) := \theta(\alpha_t(a))$ ($t \in G,\ a \in A$); note that if $\alpha$ is continuous in the point-norm topology then so is $\alpha^\theta$. We say $\alpha$ is a \emph{$\theta$-action} if $\alpha^\theta$ extends to a weak*-continuous automorphism of $\theta(A)''$ such that the map $t \mapsto \alpha^\theta_t(x)$ is weak*-continuous for each $x \in \theta(A)''$.
We will need to work with $\overline{\thcrs}^{\rm w^*}$, which we denote by $\wstthcrs$.

Let $M$ be a von~Neumann algebra on a Hilbert space $\HH$, and $\beta : G \to \Aut(M)$ a group homomorphism which is continuous in the point-weak* topology; then the triple $(M,G,\beta)$ is called a \emph{$W^*$-dynamical system}. 
Defining representations $\pi$ and $\lambda$ of $M$ and $G$ respectively on $L^2(G,\HH)$ by the same formulae as above gives a covariant pair of representations $(\pi , \lambda)$ of $(M , G ,\beta)$, with $\pi$ normal. The \emph{(von~Neumann) crossed product} of $(M , G ,\beta)$, denoted $\vNcros{M}{G}{\beta}$, is the von~Neumann algebra generated by $\pi(M)$ and $\lambda(G)$ on $L^2(G,\HH)$. 
See Takesaki~\cite[Chapter X]{Tak03} for more on this construction.

Classically, $u : G \to \CC$ is called a Herz--Schur multiplier if $u$ is a completely bounded multiplier of the Fourier algebra of $G$ (the Fourier algebra of $G$, \falg, will be defined in Section~\ref{sec:fourieralgebras}) {\it i.e.}\ $uv \in \falg$ for all $v \in \falg$ and the map
\[
    m_u : A(G) \to A(G) ;\ m_u(v) := uv , \quad v \in A(G) ,
\]
is completely bounded; the space of such functions is denoted $\Mcb\falg$.
Bo\.zejko--Fendler~\cite{BF84} discuss several equivalent definitions of Herz--Schur multipliers, including: Herz--Schur multipliers on $G$ coincide with the completely bounded multipliers of \vN. One can further show that if $u$ is a Herz--Schur multiplier of $G$ then $m_u^* : \vN \to \vN$ leaves \rgCst\ invariant.
In defining Herz--Schur $(A,G,\alpha)$-multipliers we took the reverse approach, defining first a map on \rcrs. If the dynamical system in question is $(\CC,G,1)$ then the corresponding crossed product is precisely \rgCst, so (identifying $\CB[\CC]$ with \CC) we have that $u$ is a Herz--Schur ($\CC,G,1$)-multiplier if and only if $u$ is a Herz--Schur multiplier.
The goal of Section~\ref{sec:fourieralgebras} is to introduce a space for a $C^*$-dynamical system $(A,G,\alpha)$ which generalises the Fourier algebra of a locally compact group, and identify Herz--Schur $(A,G,\alpha)$-multipliers with the completely bounded `multipliers' of this space.
Unlike the classical case it is not clear if the map $S_F$ corresponding to $F \in \HSm$ extends to the weak*-closure of \rcrs, so we make the following definition.

\begin{definition}\label{de:thetamultiplier}
Let $(\theta , \HH_\theta)$ be a faithful representation of $A$ on a separable Hilbert space.
A bounded function $F : G \to \CB[A]$ will be called a \emph{$\theta$-multiplier} of $(A,G,\alpha)$ if the map
\[
    S_F^\theta : \pi^\theta(a)\lambda^\theta_t \mapsto \pi^\theta \big( F(t)(a) \big) \lambda^\theta_t , \quad a \in A,\ t \in G,
\]
has an extension to a completely bounded weak*-continuous map on \wstthcrs.
\end{definition}

\noindent
Note that \cite[Remark 3.4]{MTT18} shows that Herz--Schur $\theta$-multipliers of $(A,G,\alpha)$ act in the same way as Herz--Schur $(A,G,\alpha)$-multipliers, when viewed through a weak*-continuous functional.
To simplify notation I will often omit the superscript $\theta$ from the maps $S_F$ associated to the multipliers defined above; it will be clear from the presence/absence of $\theta$ elsewhere in the notation where $S_F$ is acting.

The following result \cite[Theorem 3.8]{MTT18} provides a useful characterisation of Herz--Schur $(A,G,\alpha)$-multipliers, generalising the classical transference theorem (see {\it e.g.}\ \cite{BF84}).

\begin{theorem}\label{th:transference}
Let $(A,G,\alpha)$ be a $C^*$-dynamical system with $A \subseteq \Bd$, and let $F : G \to \CB$ be a bounded, pointwise-measurable, function.
The following are equivalent:
\begin{enumerate}[i.]
    \item $F$ is a Herz--Schur $(A,G,\alpha)$-multiplier;
    \item there exist a separable Hilbert space $\HH_\rho$, a non-degenerate representation $\rho : A \to \Bd[\HH_\rho]$ and $V,W \in L^\infty(G,\Bd[\HH,\HH_\rho])$ such that 
    \[
        \mathcal{N}(F)(s,t)(a) := \alpha_{t\inv} ( F(ts\inv) ( \alpha_t(a) ) ) = W(t)^* \rho(a) V(s) .
    \]
\end{enumerate}
\end{theorem}

\section{Fourier space of a crossed product} \label{sec:fourieralgebras}

\noindent
In this section we develop a space for the crossed product which is analogous to the Fourier algebra in the setting of group \CA s and \vNA s, and study the multipliers of this space. To motivate this discussion and fix notation let us first recall some facts about the Fourier algebra of a locally compact group $G$.
The Fourier algebra of $G$, introduced by Eymard~\cite{Eym64}, denoted \falg, is the space of coefficients of the left regular representation; that is, the space of functions $u : G\to \CC$ of the form
\[
    u(t) = \ip{\lambda^G_t\xi}{\eta} , \quad t\in G,\ \xi,\eta\in L^2(G) .
\]
The linear space defined in this way becomes an algebra under pointwise multiplication, and turns out to be the predual of the group \vNA\ \vN. 
Bo\.{z}ejko--Fendler~\cite{BF84} proved that the space $\Mcb \falg$ is isometrically isomorphic to the space of Herz--Schur multipliers of $G$, so they are treated as the same space.

Recall that $A$ denotes a unital \CA\ and $\alpha : G \to \Aut(A)$ is a point-norm continuous homomorphism.
The following definition is adapted from Pedersen~\cite[7.7.4]{Ped79}.

\begin{definition}\label{de:fourieralgebraofsystem}
Let $(A,G,\alpha)$ be a $C^*$-dynamical system and let $(\theta,\HH_\theta)$ be a faithful representation of $A$. 
Let $\tilde{u} \in (\thcrs)^*$ be a functional of the form
\begin{equation}\label{eq:defofelementofcrossedprodFourieralg}
    \tilde{u}(T) = \sum_{n \in \NN} \ip{T \xi_n}{\eta_n} ,\quad T \in \thcrs,
\end{equation}
where $\xi_n,\eta_n \in L^2(G,\HH_\theta)$ satisfy $\sum_n\nrm[\xi_n]^2 < \infty ,\ \sum_n\nrm[\eta_n]^2 < \infty$. The set of such functionals forms a linear space which can be identified with $((\thcrs)^{''})_*$. 
To each such $\tilde{u}$ we associate the function $u : G\to A^*$ defined by
\begin{equation}\label{eq:crossedFalgactingGtoAstr}
    u(t)(a) := \tilde{u} \big( \pi^\theta(a)\lambda^\theta_t \big), \quad a \in A,\ t \in G.
\end{equation}
The set of all functions from $G$ to $A^*$ associated to functionals of the form of $\tilde{u}$ is a linear space (with the obvious operations), which we again identify with the predual of $(\thcrs)^{''}$ and endow with the norm 
\[
    \nrm[u]_\Aa := \nrm[\tilde{u}] ,
\]
where the right side means the norm of $\tilde{u}$ as a member of the dual space of $(\thcrs)''$. 
The resulting space is called the \emph{Fourier space of $(A,G,\alpha)$} and denoted \cfal\ (when $\theta = \id$ we write $\cfalg{A}{G}{\alpha}$).
\end{definition}

\noindent
In the case of the system $(\CC,G,1)$ the only representation $\theta$ of $\CC$ is trivial, $\pi^\theta$ also becomes trivial, and we can identify $\lambda^\theta$ with $\lambda^G$; thus the above definition gives the predual of $(\rcros{\CC}{G}{1})^{''} \cong \vN$, so the space defined may be identified with \falg.
Definition~\ref{de:fourieralgebraofsystem} also works unchanged for a $W^*$-dynamical system $(M,G,\beta)$; in this case the definition identifies the predual of the \vNA\ $\vNcros{M}{G}{\beta}$ with the space of functions $u : G \to M_*$ corresponding to functionals of the form (\ref{eq:defofelementofcrossedprodFourieralg}) \cite{Tak75}. 
The following is shown by Fujita~\cite[Lemma 3.4]{Fuj79}. \

\begin{remark}\label{re:compactsupportdenseinfourierspace}
{\rm 
Let $(A,G,\alpha)$ be a $C^*$-dynamical system and $(\theta , \HH_\theta)$ a faithful representation of $A$. The compactly supported functions form a dense subset of $\cfal$. The same holds for a $W^*$-dynamical system.
}
\end{remark}
%

It appears that the space $\cfal$ was first defined for $W^*$-dynamical systems and their crossed products by Takai~\cite{Tak75}.
Note that in the case of a $W^*$-dynamical system Fujita~\cite{Fuj79} introduces a Banach algebra structure on \cfal, but we do not pursue this here.

We now define multipliers of the Fourier space of a $C^*$-dynamical system, and study the relationship with Herz--Schur multipliers of the system. The results in this section are essentially predual versions of some results in \cite[Section 3]{MTT18}.

\begin{definition}
A bounded 
function $F : G \to \Bd[A]$ is called a \emph{multiplier of \cfal} if there is a bounded map
\[
    s_F : \cfal\ \to \cfal
\]
such that
\[
    (s_Fu)(t)(a) = u(t) \big( F(t)(a) \big), \quad u \in \cfal ,\ t \in G,\ a \in A .
\]
The norm of a multiplier $F$ is defined by $\nrm[F]_\Mult := \nrm[s_F^*]$.
If moreover $F$ maps into $\CB$ and $s_F^*$ is completely bounded then $F$ is called a \emph{completely bounded multiplier of \cfal}. In this case the completely bounded multiplier norm of $F$ is defined by $\nrm[F]_{\Mcb} := \nrm[s_F^*]_\cb$.
The spaces of bounded and completely bounded multipliers of \cfal\ are denoted $\Mult \cfal$ and $\Mcb \cfal$ respectively.
\end{definition}


\begin{lemma}\label{le:multiplieroffourieralgiffmultiplierofsystem}
Let $F : G \to \Bd[A]$ be a bounded, pointwise-measurable, function, and $(\theta , \HH_\theta)$ be a faithful representation of $A$. 
The following are equivalent:
\begin{enumerate}[i.]
    \item $F$ is a multiplier of \cfal;
    \item there is an ultraweakly continuous bounded operator $S_F$ on $(\thcrs)^{''}$ such that $S_F(\pi^\theta(a)\lambda^\theta_t) = \pi^\theta(F(t)(a))\lambda^\theta_t$ for all $a\in A,\ t \in G$.
\end{enumerate}
Moreover, if either condition holds then $\nrm[F]_\Mult = \nrm[S_F]$. Finally, $F$ is a completely bounded multiplier of $\cfal$ if and only if the map $S_F$ of (ii) is completely bounded, and in this case $\nrm[F]_\Mcb = \nrm[S_F]_\cb$.
\end{lemma}
\begin{proof}
If $F$ is a multiplier of \cfal\ then $S_F := s_F^*$ is the required map because for any $u\in \cfal$
\[
    \dualp{S_F(\pi^\theta(a)\lambda^\theta_t)}{u} = \dualp{\pi^\theta(a)\lambda^\theta_t}{s_Fu} = u(t) \big(F(t)(a) \big) = \dualp{\pi^\theta \big( F(t)(a) \big)\lambda^\theta_t}{u}.
\]

Conversely, given $u \in \cfal$, the function
\[
    \pi^\theta(a)\lambda^\theta_t \mapsto \dualp{S_F \big(\pi^\theta(a)\lambda^\theta_t) \big)}{u}
\]
extends to an ultraweakly continuous linear functional on $(\thcrs)^{''}$. Therefore,
there is $Fu \in \cfal$ with $\nrm[Fu] \leq \nrm[u]_\Aa\nrm[S_F]$, such that $\dualp{\pi^\theta(a)\lambda^\theta_t}{Fu} = \dualp{S_F(\pi^\theta(a)\lambda^\theta_t)}{u}$.
It follows that the map $u \mapsto Fu$ is continuous, and
\[
    (Fu)(t)(a) = \dualp{\pi^\theta(a)\lambda^\theta_t}{Fu} = \dualp{S_F \big(\pi^\theta(a)\lambda^\theta_t \big)}{u} = u(t) \big(F(t)(a) \big) ,
\]
for all $t \in G,\ a \in A$, so $F$ is a multiplier of \cfal\ with $s_Fu = Fu$ for all $u \in \cfal$.
Finally, $\nrm[F]_M = \nrm[s_F^*] = \nrm[S_F]$ by definition.
The statements about completely bounded multipliers follow similarly.
\end{proof}

\noindent
Since the ultraweak topology on $(\thcrs)^{''}$ is the relative ultraweak topology from $\Bd[L^2(G)\otimes \HH_\theta]$ 
we consider the map $S_F$ of the previous lemma to be a weak*-continuous map on \wstthcrs.

\begin{corollary}\label{co:HSmultsarecbmultsofFourieralg}
The space of Herz--Schur $\theta$-multipliers of $(A,G,\alpha)$ coincides isometrically with the space of completely bounded multipliers of \cfal.
\end{corollary}
\begin{proof}
Immediate from Lemma~\ref{le:multiplieroffourieralgiffmultiplierofsystem} and \cite[Corollary 3.10]{MTT18}.
\end{proof}

In the next section we will use the description of Herz--Schur multipliers of a dynamical system as completely bounded multipliers of the Fourier space in studying weak amenability of the system.

\begin{remark}\label{re:connectionwithBedosConti}
{\rm B\'{e}dos and Conti~\cite[Section 4]{BC15} have taken a Hilbert $C^*$-module approach to completely bounded multipliers of a discrete (twisted) $C^*$-dynamical system.
It is easy to check that $F : G \to \CB$ is a Herz--Schur $(A,G,\alpha)$-multiplier if and only if $T_F : G \times A \to A;\ T_F(t,a) := F(t)(a)$ ($t \in G,\ a \in A$) is a completely bounded reduced multiplier of $(A,G,\alpha)$, in the sense of B\'{e}dos--Conti.
The same authors have also introduced a version of the Fourier--Stieltjes algebra for discrete (twisted) $C^*$-dynamical systems, again using Hilbert $C^*$-modules~\cite{BC16}.
}
\end{remark}

\section{Weak amenability}
\label{sec:weakamenability}

\noindent
In this section we define weak amenability of a $C^*$-dynamical system; when the group is discrete we show this is equivalent to the CBAP of the reduced crossed product.
We also define weak amenability of a $W^*$-dynamical system, and when the group is discrete show this is equivalent to the weak* CBAP of the associated crossed product.
The weak* CBAP for crossed products of $W^*$-dynamical systems has been studied by Haagerup--Kraus~\cite[Section 3]{HKr94}; they showed that if $(M,G,\alpha)$ is a $W^*$-dynamical system with $G$ weakly amenable and $M$ having the weak* CBAP then it is not true in general that $\vNcros{M}{G}{\alpha}$ has the weak* CBAP. We will give an example of an assumption under which this implication does hold.
The CBAP for the reduced crossed product of a $C^*$-dynamical system has been studied by Sinclair--Smith~\cite{SS97} under the assumption that the group is amenable; here we give some other conditions under which the reduced crossed product has the CBAP.

As before $A$ is a unital \CA\ and $(\theta,\HH_\theta)$ is a faithful representation of $A$.
In this section $G$ will always denote a discrete group. 
Denote by $\alpha : G \to \Aut(A)$ a homomorphism, so that $(A,G,\alpha)$ is a $C^*$-dynamical system.
Since $G$ is discrete there is a canonical conditional expectation $\EE^\theta : \rcros{\theta(A)}{G}{\alpha^\theta} \to \theta(A)$ which is equivariant (see Brown--Ozawa~\cite[Proposition 4.1.9]{BrO08}). 
We denote by $\EE$ the completely positive map defined by
\[
    \thcrs \cong \rcros{\theta(A)}{G}{\alpha^\theta} \to A;\ \sum_{t \in G} \pi^\theta(a_t)\lambda^\theta_t \mapsto a_e , \quad a_t \in A .
\]
The triple $(M,G,\beta)$ will denote a discrete $W^*$-dynamical system, {\it i.e.}\ $M$ is a von Neumann algebra acting on a Hilbert space $\HH_M$, $G$ is a discrete group, and $\beta : G \to \Aut(M)$ a homomorphism.
The symbol $\EE$ will also be used for the conditional expectation $\vNcros{M}{G}{\beta} \to M$, defined similarly.

Our main questions are:
\begin{itemize}
    \item For a $C^*$-dynamical system $(A,G,\alpha)$ what are necessary and/or sufficient conditions for $\thcrs$ to have the completely bounded approximation property?
    \item For a $W^*$-dynamical system $(M,G,\beta)$ what are necessary and/or sufficient conditions for $\vNcros{M}{G}{\beta}$ to have the weak* completely bounded approximation property?
\end{itemize}
Our approach to these problems is to consider certain Herz--Schur multipliers of the system in question. 
Since we have so far only considered Herz--Schur multipliers of a $C^*$-dynamical system we briefly describe a construction, mentioned by Fujita~\cite[page 56]{Fuj79}, which shows that Herz--Schur multipliers of a $W^*$-dynamical system are particular cases of the weak*-extendable multipliers of Definition~\ref{de:thetamultiplier}.
For the $W^*$-dynamical system $(M,G,\beta)$, where $M$ is a \vNA\ on the separable Hilbert space $\HH_M$, consider the set
\[
    M_\beta := \{ x \in M : t \mapsto \beta_t(x)\ \text{is norm-continuous for all $t \in G$} \} .
\]
Then $M_\beta$ is a $G$-invariant, weak*-dense $C^*$-subalgebra of $M$ containing the identity, and $(M_\beta , G,\beta)$ is a $C^*$-dynamical system, with $M_\beta$ faithfully represented on $\Bd[\HH_M]$. The construction of the reduced crossed product $\rcros{M_\beta}{G}{\beta}$, using the faithful representation $\id: M_\beta \to \Bd[\HH_M]$, gives a weak*-dense $C^*$-subalgebra of $\vNcros{M}{G}{\beta}$.
It follows that $\cfalg[\id]{M_\beta}{G}{\beta}$ can be identified with the predual of $\vNcros{M}{G}{\beta}$, and that the Herz--Schur $\id$-multipliers of $(M_\beta ,G,\beta)$ are completely bounded multipliers of $\cfalg[\id]{M_\beta}{G}{\beta}$ and the associated maps possess completely bounded, weak*-continuous extensions to $\vNcros{M}{G}{\beta}$.

For a \CA\ $B$ let $\CBw[B]$ be the space of completely bounded maps on $B$ that extend to completely bounded, weak*-continuous, maps on $B''$.

\begin{definition}\label{de:weakamenabilityofsystem}
A $C^*$-dynamical system $(A,G,\alpha)$ will be called \emph{weakly amenable} if there exists a net $(F_i)$ of finitely supported Herz--Schur $(A,G,\alpha)$-multipliers such that $F_i(t)$ is a finite rank completely bounded map on $A$ for all $t \in G$,
\[ \label{eq:defofweakamenabilityconvergence}
    F_i(t)(a) \stackrel{\nrm}{\to} a \quad  \text{for all $t\in G,\ a\in A$,}
\]
and $\sup_i \nrm[F_i]_{\HS} = K < \infty$. The infimum of all such $K$ is denoted by $\Lambda_\cb(A,G,\alpha)$.

A $W^*$-dynamical system $(M,G,\beta)$, with $M$ acting on $\Bd[\HH_M]$, will be called weakly amenable if there is a net $F_i : G \to \CBw[M_\beta]$ of finitely supported Herz--Schur $\id$-multipliers of $(M_\beta , G, \beta)$, such that $F_i(t)$ extends to a finite rank completely bounded map on $M$ for all $t \in G$,
\begin{equation} \label{eq:defofweakamenabilityconvergenceWsystem}
    F_i(t)(a) \stackrel{w^*}{\to} a \quad  \text{for all $t\in G,\ a\in M$,}
\end{equation}
and $\sup_i \nrm[F_i]_{\HS} = K < \infty$.
\end{definition}

\noindent
Observe that if $A = \CC$ then the finite rank condition is always satisfied, so Definition~\ref{de:weakamenabilityofsystem} reduces to weak amenability of $G$.

\begin{remark}\label{re:weakamenabilityofsystemimpliesofgroupandvNsystems}
If $(A,G,\alpha)$ is a weakly amenable $C^*$-dynamical system with $A$ unital, such that $A$ is faithfully represented on a separable Hilbert space $\HH$, and the maps $F_i$ of Definition~\ref{de:weakamenabilityofsystem} satisfy
\begin{equation}\label{eq:systemimpliesgroupcondition}
    F_i(t) \circ \alpha_r = \alpha_r \circ F_i(t) , \quad r,t \in G ,
\end{equation}
then $G$ is weakly amenable.
\end{remark}
\begin{proof}
Suppose $(A,G,\alpha)$ is weakly amenable and take a net $(F_i)$ of Herz--Schur $(A,G,\alpha)$-multipliers satisfying the definition. Let $\xi \in \HH$ be a unit vector. Condition (\ref{eq:systemimpliesgroupcondition}) ensures that the map
\[
    v_i : G \to \CC;\ v_i(ts\inv) := \ip{\mathcal{N}(F_i)(s,t)(\unA)\xi}{\xi} , \quad s,t \in G
\]
is well-defined.
Let $V_i$ and $W_i$ be the maps associated to $\mathcal{N}(F_i)$ in Theorem~\ref{th:transference}. Then
\[
    v_i(ts\inv) = \ip{\mathcal{N}(F_i)(s,t)(\unA)\xi}{\xi} = \ip{V_i(s)\xi}{W_i(t)\xi} , \quad s,t \in G,
\]
Hence $v_i : G \to \CC$ is a Herz--Schur multiplier (see Bo\.zejko--Fendler~\cite{BF84}, these statements are part of the proof of \cite[Proposition 4.1]{MTT18} for a particular case where (\ref{eq:systemimpliesgroupcondition}) holds).
Since $F_i$ has finite support so does $v_i$.
We have
\[
    \nrm[v_i]_\Mcb \leq \esssup_{s \in G}\nrm[V_i(s)]\esssup_{t \in G}\nrm[W_i(t)] = \nrm[\mathcal{N}(F_i)]_\mathfrak{S} = \nrm[F_i]_\HS .
\]
Since
\[
    v_i(ts\inv) = \ip{\mathcal{N}(F_i)(s,t)(\unA)\xi}{\xi} = \ip{F_i(ts\inv)(\unA) \xi}{\xi} \to \ip{\unA \xi}{\xi} = 1 ,
\]
$G$ is weakly amenable.
\end{proof}

We now prove our characterisation of weak amenability for $C^*$-dynamical systems.
Since the reduced crossed product $C^*$-algebra and the collection of Herz--Schur $(A,G,\alpha)$-multipliers do not depend on the representation $\theta$ of $A$ we will omit $\theta$ from our notation, working with a fixed representation of $A$ on a separable Hilbert space $\HH$.

\begin{theorem}\label{th:systemweaklyamenableiffnormclosedcrossedprodhasCBAP}
Let $(A,G,\alpha)$ be a $C^*$-dynamical system, with $G$ a discrete group and $A$ a unital \CA. 
The following are equivalent:
\begin{enumerate}[i.]
    \item $(A,G,\alpha)$ is weakly amenable;
    \item $\rcrs$ has the completely bounded approximation property.
\end{enumerate}
Moreover, if the conditions hold then $\Lambda_\cb(A,G,\alpha) = \Lambda_\cb(\rcrs)$.
\end{theorem}
\begin{proof}
(i)$\implies$(ii) Suppose that $(F_i)$ is a net of Herz--Schur $(A,G,\alpha)$-multipliers satisfying weak amenability of the system. It follows immediately that the net $(S_{F_i})$ of corresponding maps on \rcrs\ consists of completely bounded, finite rank, maps satisfying $\sup \nrm[S_{F_i}]_\text{cb} \leq C < \infty$. It remains to show that $\nrm[S_{F_i}(T) - T] \to 0$ for all $T \in \rcrs$. 
For this, it suffices to show that $\nrm[S_{F_i}(\sum_t \pi(a_t)\lambda_t) - \sum_t \pi(a_t)\lambda_t] \to 0$ when the sums are finite. Indeed, for any $T \in \rcrs$ and $\epsilon >0$, we can find $a_t \in A$ with $\nrm[T - \sum_t \pi(a_t)\lambda_t] < \epsilon$, where only a finite number of $a_t$ are non-zero, so
\[
\begin{split}
    \nrm[S_{F_i}(T) - T] &\leq \bignrm[S_{F_i}(T) - S_{F_i} \Big( \sum_t \pi(a_t)\lambda_t \Big)] \\
        &\quad + \bignrm[S_{F_i} \Big( \sum_t \pi(a_t)\lambda_t \Big) - \sum_t \pi(a_t)\lambda_t]
         + \bignrm[\sum_t \pi(a_t)\lambda_t - T] \\
        &< C \epsilon + \bignrm[S_{F_i} \Big( \sum_t \pi(a_t)\lambda_t \Big) - \sum_t \pi(a_t)\lambda_t] + \epsilon .
\end{split}
\]
Now
\[
\begin{split}
    \bignrm[S_{F_i} \Big( \sum_t \pi(a_t)\lambda_t \Big) - \sum_t \pi(a_t)\lambda_t] &= \bignrm[\sum_t \pi \big( F_i(t)(a_t) \big)\lambda_t - \sum_t \pi(a_t)\lambda_t] \\
        &\leq \sum_t \nrm[\pi \big( F_i(t)(a_t) - a_t \big)\lambda_t] \to 0
\end{split}
\]
as $F_i(t)(a) \to a$ for all $a \in A,\ t \in G$.
It follows that $\Lambda_\cb(\rcrs) \leq \Lambda_\cb(A,G,\alpha)$.

(ii)$\implies$(i) We will use a similar idea to Haagerup~\cite[Lemma 2.5]{Haaun16}. 
First consider a finite rank, completely bounded, map $\rho : \rcrs \to \rcrs$. Take $T_1 , \ldots , T_k \in \rcrs$ which span $\ran \rho$, so there are $\phi_1, \ldots ,\phi_k \in (\rcrs)^*$ such that
\[
    \rho = \sum_{j=1}^k \phi_j \otimes T_j ,
\]
where $(\phi_j \otimes T_j) (T) = \phi_j(T) T_j$ ($T \in \rcrs$).
We note that, for a matrix $(x_{p,q}) \in M_n(\rcrs)$,
\[ \begin{split}
    \bignrm[\Big( \sum_{j=1}^k \phi_j \otimes T_j \Big)^{(n)} (x_{p,q})] \leq \sum_{j=1}^k \nrm[(\phi_j \otimes T_j )^{(n)} (x_{p,q})] 
        &= \sum_{j=1}^k \nrm[\phi_j^{(n)}(x_{p,q})\ \diag_n(T_j)] \\
        &\leq \sum_{j=1}^k \nrm[\phi_j] \nrm[(x_{p,q})] \nrm[T_j] ,
\end{split} \]
where $\diag_n(T)$ denotes the diagonal $n \times n$ matrix with each diagonal entry equal to $T$. Thus
\begin{equation}\label{eq:cbnormfiniterank}
    \bignrm[\sum_{j=1}^k \phi_j \otimes T_j]_\cb \leq \sum_{j=1}^k\nrm[\phi_j] \nrm[T_k] .
\end{equation}
For each $j$ and each $n \in \NN$ find $a_{j,n}^i \in A$ and $s_{j,n}^i \in G$ such that $T_{j,n} := \sum_{i=1}^{k_{j,n}} \pi(a_{j,n}^i)\lambda_{s_{j,n}^i}$ satisfies $\nrm[T_j - T_{j,n}] < 1/(nk\max_j\nrm[\phi_j])$.
Define $\rho_n := \sum_{j=1}^k \phi_j \otimes T_{j,n}$.
Then
\begin{equation}\label{eq:differencecbnormafteradjusting} \begin{split}
    \nrm[\rho - \rho_n]_\cb = \bignrm[\Big( \sum_{j=1}^k \phi_j \otimes T_j \Big) - \Big( \sum_{j=1}^k \phi_j \otimes T_{j,n} \Big)]_\cb 
        &\leq \sum_{j=1}^k \nrm[\phi_j \otimes (T_j - T_{j,n})]_\cb \\
        &\leq \sum_{j=1}^k \nrm[\phi_j] \nrm[T_j - T_{j,n}] < \frac{1}{n} .
\end{split} \end{equation}

Now let $(\rho_\gamma)$ be a net of maps on \rcrs\ satisfying the conditions of the CBAP. By the above procedure we obtain a net of maps $(\rho'_{\gamma,n})$ on $\rcrs$ which are finite rank, with range in $\lspan \{ \pi(a)\lambda_t : a \in A,\ t \in G \}$. 
It is easily checked that $\rho'_{\gamma,n} \to \id$ in point-norm, using the product directed set.
As in (\ref{eq:cbnormfiniterank}) we have that each $\rho'_{\gamma,n}$ is completely bounded; by (\ref{eq:differencecbnormafteradjusting}) we have $\nrm[\rho_\gamma - \rho'_{\gamma , n}]_\cb < 1/n$ for all $\gamma$ and all $n \in \NN$, so $\nrm[\rho'_{\gamma,n}]_\cb < \nrm[\rho_\gamma]_\cb + 1/n$. 
Let $C = \sup_\gamma \nrm[\rho_\gamma]_\cb$ and define
\[
    \rho_{\gamma,n} := \frac{C}{C + 1/n} \rho'_{\gamma,n} ,
\]
so that $(\rho_{\gamma ,n})$ is a net satisfying the CBAP for $\rcrs$, uniformly bounded by $C$, and with range in $\lspan \{ \pi(a)\lambda_t : a \in A,\ t \in G \}$.
Define $F_{\gamma,n} : G \to \CB$ by
\begin{equation}\label{eq:HSmultassoctoCBAPmaps}
    F_{\gamma,n}(t)(a) := \EE \big( \rho_{\gamma,n}(\pi(a)\lambda_t) \lambda_t^* \big) , \quad a \in A,\ t \in G .
\end{equation}
It is easy to see that $\supp F_{\gamma,n} \subseteq \{ s^i_{j,n} : 1 \leq i \leq k_{j,n},\ 1 \leq j \leq k \}$.
As $\rho_{\gamma,n}$ is finite rank, with range spanned by finite sums of elements of the form $\pi(a)\lambda_r$ ($a \in A,\ r \in G$), it follows that each $F_{\gamma,n}(t)$ is a finite rank map on $A$, with $\ran F_{\gamma,n}(t) \subseteq \lspan \{ a \in A : \pi(a)\lambda_t \in \ran \rho_{\gamma,n} \}$.
Since $\rho_{\gamma,n} \to \id$ in point-norm we have, for all $t \in G,\ a \in A$,
\[
    F_{\gamma,n}(t)(a) =  \Big( \EE \big(\rho_{\gamma,n}(\pi(a)\lambda_t) \lambda_t^* \big) \Big) \to \EE \big( \pi(a)\lambda_{t t\inv} \big) = a .
\]
It remains to show that each $F_{\gamma ,n}$ is a Herz--Schur $(A,G,\alpha)$-multiplier and $\nrm[S_{F_{\gamma,n}}]_\text{cb} = \nrm[\rho_{\gamma,n}]_\text{cb}$. 
Write the completely bounded maps $\rho_{\gamma ,n}$ as  $\rho_{\gamma ,n}( \cdot ) = W_{\gamma ,n}^* \Psi_{\gamma ,n}( \cdot ) V_{\gamma ,n}$, where $V_{\gamma ,n},W_{\gamma ,n}$ are bounded operators and $\Psi_{\gamma ,n}$ is a representation.
To see that $F_{\gamma ,n}$ is a Herz--Schur $(A,G,\alpha)$-multiplier calculate
\[
\begin{split}
    \mathcal{N}(F_{\gamma ,n})(s,t)(a) &= \alpha_{t\inv} \Big( \EE \big( \rho_{\gamma ,n}(\pi(\alpha_t(a)) \lambda_{t s\inv}) \lambda_{s t\inv} \big) \Big) \\
        &= \alpha_{t\inv} \Big( \EE \big( \rho_{\gamma ,n}(\lambda_t \pi(a) \lambda_{s\inv}) \lambda_{s t\inv} \big) \Big) \\
        &= \EE \big( \lambda_{t\inv} \rho_{\gamma ,n}(\lambda_t \pi(a) \lambda_{s\inv}) \lambda_s \big) \\ 
        &= \EE \big( \lambda_{t\inv} W_{\gamma ,n}^* \Psi_{\gamma ,n}(\lambda_t) \Psi_{\gamma ,n} (\pi(a)) \Psi_{\gamma ,n}(\lambda_{s\inv}) V_{\gamma ,n} \lambda_s \big) \\
        &= U^* \lambda_{t\inv} W_{\gamma ,n}^* \Psi_{\gamma ,n}(\lambda_t) \Psi_{\gamma ,n} (\pi(a)) \Psi_{\gamma ,n}(\lambda_{s\inv}) V_{\gamma ,n} \lambda_s U \\
        &= \mathcal{W}_{\gamma ,n}(t)^* \Psi_{\gamma ,n}(\pi(a)) \mathcal{V}_{\gamma ,n}(s) ,
\end{split}
\]
where $U : \HH \to \ell^2(G) \otimes \HH ;\ \xi \mapsto \delta_e \otimes \xi$, and
\[
    \mathcal{V}_{\gamma ,n}(s) := \Psi_{\gamma ,n}(\lambda_{s\inv}) V_{\gamma ,n} \lambda_s U , \quad \mathcal{W}_{\gamma ,n}(t) := \Psi_{\gamma ,n}(\lambda_{t\inv}) W_{\gamma ,n} \lambda_t U ,
\]
so $F_{\gamma ,n}$ is a Herz--Schur $(A,G,\alpha)$-multiplier by Theorem~\ref{th:transference}.

For the norm equality let $(e_l)_\Lambda$ be an orthonormal basis for $\HH$,
\[
    V : \ell^2(G) \otimes \HH \to \ell^2(G) \otimes \ell^2(G) \otimes \HH ;\ \delta_g \otimes e_l \mapsto \delta_g \otimes \delta_g \otimes e_l ,
\]
where $\{ \delta_g : g \in G \}$ denotes the canonical orthonormal basis for $\ell^2(G)$, and let $\tau$ denote the coaction
\[
    \tau : \rcrs \to \rgCst \otimes_\text{min} \rcrs ;\ \pi(a)\lambda_t \mapsto \lambda^G_t \otimes \pi(a)\lambda_t ,
\]
for all $a \in A,\ t \in G$.
We claim
\begin{equation}\label{eq:dilationformofSFfornormequality}
    S_{F_{\gamma,n}}(x) = V^* (\id \otimes \rho_{\gamma,n})\tau(x) V , \quad x \in \rcrs ,
\end{equation}
which implies $S_{F_{\gamma ,n}}$ is completely bounded, with $\nrm[S_{F_{\gamma,n}}]_\mathrm{cb} = \nrm[\rho_{\gamma,n}]_\mathrm{cb}$.
To prove the claim we first assume $\rho_{\gamma,n}$ has one-dimensional range generated by $\pi(b)\lambda_r$ for some $b \in A,\ r \in G$.
Then, for $x,y \in G,\ l,m \in \Lambda$,
\[\begin{split}
    & \ip{V^* (\id \otimes \rho_{\gamma,n})\tau \big( \pi(a)\lambda_t \big) V(\delta_x \otimes e_m)}{\delta_y \otimes e_l} \\
    & \qquad \qquad = \ip{\lambda_t \otimes \rho_{\gamma,n} \big( \pi(a)\lambda_t \big)(\delta_x \otimes \delta_x \otimes e_m)}{\delta_y \otimes \delta_y \otimes e_l} \\
            & \qquad \qquad= \ip{\delta_{tx}}{\delta_y}\ip{\rho_{\gamma,n} \big( \pi(a)\lambda_t \big)(\delta_x \otimes e_m)}{\delta_y \otimes e_l} \\
            & \qquad \qquad= \ip{\delta_{tx}}{\delta_y}\ip{\pi(b)\lambda_r(\delta_x \otimes e_m)}{\delta_y \otimes e_l} \\
            & \qquad \qquad= \ip{\delta_{tx}}{\delta_y}\ip{\pi(b)\lambda_r(\delta_x \otimes e_m)(y)}{e_l} \\
            & \qquad \qquad= \ip{\delta_{tx}}{\delta_y}\ip{\alpha_{y\inv}(b)e_m}{e_l}\ip{\delta_{rx}}{\delta_y} .
\end{split}\]
On the other hand,
\[\begin{split}
   & \ip{S_{F_{\gamma,n}} \big( \pi(a)\lambda_t \big) (\delta_x \otimes e_m)}{\delta_y \otimes e_l} \\
            & \qquad = \ip{\pi \big(F_{\gamma,n}(t)(a) \big)\lambda_t(\delta_x \otimes e_m)}{\delta_y \otimes e_l} \\
            & \qquad = \ip{\pi \Big( \EE \big( \rho_{\gamma,n}(\pi(a)\lambda_t) \lambda_{t\inv} \big) \Big) \lambda_t(\delta_x \otimes e_m)}{\delta_y \otimes e_l} \\
            & \qquad = \ip{\pi \Big( \EE \big( \pi(b) \lambda_{r t\inv} \big) \Big) \lambda_t(\delta_x \otimes e_m)}{\delta_y \otimes e_l} \\
            & \qquad = \ip{\delta_r}{\delta_t}\ip{\pi(b)\lambda_t(\delta_x \otimes e_m)}{\delta_y \otimes e_l} \\
            & \qquad = \ip{\delta_r}{\delta_t}\ip{\alpha_{y\inv}(b)e_m}{e_l}\ip{\delta_{tx}}{\delta_y} .
\end{split}\]
It follows that $V^* (\id \otimes \rho_{\gamma,n})\tau ( \pi(a)\lambda_t ) V = S_{F_{\gamma,n}} ( \pi(a)\lambda_t )$.
By linearity and continuity we obtain (\ref{eq:dilationformofSFfornormequality}) when $\rho_{\gamma,n}$ has one-dimensional range. The linearity of the inner product then implies that (\ref{eq:dilationformofSFfornormequality}) holds in the general case that $\rho_{\gamma ,n}$ takes values in $\lspan \{ \pi(b_i)\lambda_{r_i} : i=1, \ldots , k \}$.
The equality $\nrm[S_{F_{\gamma ,n}}]_\text{cb} = \nrm[\rho_{\gamma ,n}]_\mathrm{cb}$ follows, so $(F_{\gamma ,n})$ is a net satisfying weak amenability of $(A,G,\alpha)$.
It also follows that $\Lambda_\cb(A,G,\alpha) \leq \Lambda_\cb(\rcrs)$.
\end{proof}

\begin{remark}\label{re:CHconstantsdegeneratesystems}
{\rm The constant $\Lambda_\cb$ introduced in Definition~\ref{de:weakamenabilityofsystem} reduces to the familiar constants defined in Section~\ref{sec:intro} in degenerate cases. 
Indeed, if $G$ is a discrete group such that the system $(\CC ,G,1)$ is weakly amenable then $G$ is weakly amenable by Remark~\ref{re:weakamenabilityofsystemimpliesofgroupandvNsystems} or Theorem~\ref{th:systemweaklyamenableiffnormclosedcrossedprodhasCBAP}; moreover, by Theorem~\ref{th:systemweaklyamenableiffnormclosedcrossedprodhasCBAP},
\[
    \Lambda_\cb(\CC ,G,1) = \Lambda_\cb(\rcros{\CC}{G}{1}) = \Lambda_\cb(\rgCst) = \Lambda_\cb(G) .
\]

Similarly, if the $C^*$-dynamical system $(A, \{ e \} , 1)$ is weakly amenable then
\[
    \Lambda_\cb(A, \{ e \} , 1) = \Lambda_\cb(\rcros{A}{\{ e \}}{1}) = \Lambda_\cb(A) .
\]
In fact, Sinclair--Smith~\cite[Theorem 3.4]{SS97} have shown that for an amenable discrete group $G$, $\Lambda_\cb(\rcros{A}{G}{\alpha}) = \Lambda_\cb(A)$, so when $(A,G,\alpha)$ is a discrete $C^*$-dynamical system with $G$ amenable we have
\[
    \Lambda_\cb(A,G,\alpha) = \Lambda_\cb(\rcros{A}{G}{\alpha}) = \Lambda_\cb(A) .
\] }
\end{remark}

We now turn to characterising weak amenability of $W^*$-dynamical systems.

\begin{lemma}\label{le:equivalentconvergenceforHSmults}
Let $(M,G,\beta)$ be a $W^*$-dynamical system, with $G$ a discrete group, and $(F_i)$ a net of Herz--Schur $\id$-multipliers of the underlying $C^*$-dynamical system $(M_\beta ,G,\beta)$.
The following are equivalent:
\begin{enumerate}[i.]
    \item $F_i(t)(a) \stackrel{w^*}{\to} a $ for all $t \in G,\ a \in M$ (equation (\ref{eq:defofweakamenabilityconvergenceWsystem}) above);
    \item $s_{F_i}u \to u$ in $\cfalg{M}{G}{\beta}$ for all $u \in \cfalg{M}{G}{\beta}$.
\end{enumerate}
\end{lemma}
\begin{proof}
(i)$\implies$(ii) By Remark~\ref{re:compactsupportdenseinfourierspace} finitely supported functions are dense in $\cfalg{M}{G}{\beta}$, so it suffices to prove the claim for singly supported $u \in \cfalg{M}{G}{\beta}$. Suppose $u \in \cfalg{M}{G}{\beta}$ is supported on $\{ s \}$ and $u(t)(a) = \sum_{n=1}^\infty \ip{\pi(a) \lambda_t \xi_n}{\eta_n}$ ($t \in G,\ a \in M$) for some families satisfying $\sum_{n=1}^\infty \nrm[\xi_n]^2 < \infty$ and $\sum_{n=1}^\infty \nrm[\eta_n]^2 < \infty$. Since $\lambda_s$ is an isometry it follows that the functional in $\pi(M)_*$ given by $\pi(a) \mapsto \sum_{n=1}^\infty \ip{\pi(a) \lambda_s \xi_n}{\eta_n}$ has the same norm as $u$; thus $\nrm[u(s)] = \nrm[u]_\Aa$.
Since $s_{F_i}u$ is also supported on $\{ s \}$ we have
\[
    \nrm[s_{F_i}u - u]_\Aa = \nrm[u(s) \circ F_i(s) - u(s)] = \sup_{\nrm[a] \leq 1} \left| u(s) \big( F_i(s)(a) - a \big) \right| \stackrel{i}{\to} 0 .
\]
Condition (ii) follows.

(ii)$\implies$(i) For any $a \in A,\ t \in G$ and $u \in \cfalg{M}{G}{\beta}$,
\[
    \big| \dualp{\pi \big( F_i(t)(a) \big) \lambda_t - \pi(a)\lambda_t}{u} \big| = \big| \dualp{\pi(a)\lambda_t}{s_{F_i}u} - \dualp{\pi(a)\lambda_t}{u} \big| \to 0 ,
\]
so $u(t)(F_i(t)(a)) \to u(t)(a)$. As $u$ varies $u(t)$ can take any value in $M_*$; thus $F_i(t)(a)$ converges to $a$ in the weak* topology. 
\end{proof}

\begin{theorem}\label{th:characterisationweakamenabilityandCBAPcrossedproducts}
Let $G$ be a discrete group, $M \subseteq \Bd[\HH_M]$ a \vNA\ acting on a separable Hilbert space, and $(M,G,\beta)$ a $W^*$-dynamical system.
Consider the conditions:
\begin{enumerate}[i.]
    \item $(M,G,\beta)$ is weakly amenable;
    \item $\vNcros{M}{G}{\beta}$ has the weak* completely bounded approximation property.
\end{enumerate}
Then (i)$\implies$(ii). If $G$ is weakly amenable then (i) and (ii) are equivalent.
\end{theorem}
\begin{proof}
(i)$\implies$(ii) Suppose that $(F_i)$ is a net of Herz--Schur $\id$-multipliers of the underlying $C^*$-dynamical system $(M_\beta,G,\beta)$ satisfying Definition~\ref{de:weakamenabilityofsystem}. Then the associated net of maps $(S_{F_i})$ on $\vNcros{M}{G}{\beta}$ are completely bounded, weak*-continuous, and finite rank. Finally, using the identification of $(\vNcros{M}{G}{\beta})_*$ with $\cfalg{M}{G}{\beta}$, we have for any $u \in \cfalg{M}{G}{\beta}$ and any $T \in \vNcros{M}{G}{\beta}$
\[
    \dualp{S_{F_i}T}{u} = \dualp{T}{s_{F_i}u} \to \dualp{T}{u}
\]
by Lemma~\ref{le:equivalentconvergenceforHSmults}, so $S_{F_i}T$ converges to $T$ in the weak* topology.

(ii)$\implies$(i) Suppose $\vNcros{M}{G}{\beta}$ has the weak* CBAP.
Given a finite set $E \subseteq G$, $\epsilon >0$, and a collection $\Omega \subseteq M_*$, choose $\rho : \vNcros{M}{G}{\beta} \to \vNcros{M}{G}{\beta}$ such that
\begin{equation}\label{eq:HSmultsinweakCBAP}
    F : G \to \CBw[M_\beta] ;\ F(t)(a) := \EE \big( \rho(\pi(a)\lambda_t) \lambda_{t\inv} \big) , \quad a \in M,\ t \in G
\end{equation}
satisfies $| \omega(a - F(t)(a)) | < \epsilon$ for all $a \in M,\ t \in E,\ \omega \in \Omega$. In this way we produce a net $(F_i)$, indexed by triples of the form $(E,\epsilon,\Omega)$, such that $F_i(t)(a) \to a$ in the weak* topology.
For each $t \in G$, $F(t)$ defined above is a finite rank map on $M$ as in the proof of Theorem~\ref{th:systemweaklyamenableiffnormclosedcrossedprodhasCBAP}; indeed, suppose $\rho = \sum_{j=1}^k \phi_j \otimes T_j$, where $\phi_j$ is a functional and $T_j \in \vNcros{M}{G}{\beta}$. Then
\[
    F(t)(a) = \EE \big( \rho(\pi(a)\lambda_t) \lambda_{t\inv} \big) = \sum_{j=1}^k\phi_j \big( \pi(a)\lambda_t \big) \EE (T_j \lambda_{t\inv} ) ,
\]
so that $\{ \EE (T_j \lambda_{t\inv} ) : j=1 , \ldots , k \}$ span $\ran F(t)$.
Similar calculations to those in the proof of Theorem~\ref{th:systemweaklyamenableiffnormclosedcrossedprodhasCBAP} 
show that $\nrm[S_F]_\cb = \nrm[\rho]_\cb$ and $F$ is a Herz--Schur $(M_\beta,G,\beta)$-multiplier. 
Each $S_F$ is a composition of weak*-continuous maps, so is weak*-extendable.
We have that the net $(F_i)$ satisfies all the conditions of weak amenability of $(M,G,\beta)$ except that it may not be finitely supported.
To correct this we use the assumption that $G$ is weakly amenable.
Let $(\varphi_j)$ be a net of functions on $G$ implementing weak amenability.
Define another net, indexed by the product directed set,
\[
    F_{i,j} : G \to \CBw[M] ;\ F_{i,j}(t)(a) := \varphi_j(t) F_i(t)(a) , \quad t \in G,\ a \in M ,
\]
which is a net of Herz--Schur $\id$-multipliers of $(M_\beta,G,\beta)$, with $S_{F_{i,j}} = S_{\varphi_j} \circ S_{F_i}$.
From the properties of $\varphi_j$ and $F_i$ we have that each $F_{i,j}$ is finitely supported, $F_{i,j}(t)$ is finite rank for all $t \in G$, and $F_{i,j}(t)(a)$ converges to $a$ in the weak* topology.
Finally, $\nrm[F_{i,j}]_\HS = \nrm[S_{F_{i,j}}]_\cb \leq \nrm[S_{\varphi_j}]_\cb \nrm[S_{F_i}]_\cb$, so the net is uniformly bounded.
\end{proof}

\begin{remarks}\label{re:weakstarHaagerupconstants}
{\rm
\begin{enumerate}[i.]
	\item In the proof of (ii)$\implies$(i) above we required weak amenability of $G$; to see why this requirement arose let us return to the proof of Theorem~\ref{th:systemweaklyamenableiffnormclosedcrossedprodhasCBAP}.
There we are able to approximate in norm the operators $\rho_\gamma$, which implement the CBAP of $\rcrs$, by operators $\rho_{\gamma ,n}$ with finite-dimensional range spanned by elements of the form $\pi(a) \lambda_t$, such that $\nrm[\rho_{\gamma ,n}]_\cb$ is closely related to $\nrm[\rho_\gamma]_\cb$; these estimates allowed us to identify the support and Herz--Schur norm of $F_{\gamma ,n}$.
Such norm estimates are not available in the setting of Theorem~\ref{th:characterisationweakamenabilityandCBAPcrossedproducts}, so the extra hypothesis seems to be required to use the techniques in this paper.

	\item If in the above proof we make the stronger assumption that $\Lambda_\cb(G) = 1$ then the net $(\varphi_{i,n})$ may be chosen such that $\nrm[S_{\varphi_{i,n}}]_\cb$ is uniformly bounded by 1. With this assumption on $G$ we obtain $\Lambda^{\rm vN}_\cb(M,G,\beta) \leq \Lambda_\cb(\vNcros{M}{G}{\beta})$, where $\Lambda^{\rm vN}_\cb$ is the natural weak amenability constant of a $W^*$-dynamical system. It follows that if $\Lambda_\cb(G) = 1$ we have $\Lambda^{\rm vN}_\cb(M,G,\beta) = \Lambda_\cb(\vNcros{M}{G}{\beta})$.
It would be interesting to have a characterisation of when these two weak amenability constants coincide.
\end{enumerate}
}
\end{remarks}

Suppose that $(A,G,\alpha)$ is a $C^*$-dynamical system with $G$ an amenable discrete group and $A$ a nuclear \CA.
It is well known ({\it e.g.}\ Brown--Ozawa~\cite[Theorem 4.2.6]{BrO08}) that this implies $\rcrs$ is nuclear.
It is natural to ask if this fact persists for weak amenability and the CBAP: does the CBAP for $A$ and weak amenability of $G$ imply that $\rcrs$ has the CBAP?
Haagerup--Kraus \cite[Remark 3.10]{HKr94} give an example of a $W^*$-dynamical system showing that in general this is not true, which we reproduce here as a $C^*$-dynamical system.
Both $\SL(2,\ZZ)$ and $\ZZ^2$ are weakly amenable, but their semidirect product $\ZZ^2 \rtimes_\mu \SL(2,\ZZ)$ is not \cite[page 670]{HKr94} ($\mu$ denotes the usual action of $\SL(2,\ZZ)$ on $\ZZ^2$).
Since the \CA s $\rcros{\rgCst[\ZZ^2]}{\SL(2,\ZZ)}{\mu}$ and $\rgCst[\ZZ^2 \rtimes_\mu \SL(2,\ZZ)]$ are isomorphic it follows that the crossed product of a \CA\ with the CBAP by a weakly amenable group need not have the CBAP.

Sinclair--Smith~\cite{SS97} have shown that if $G$ is amenable and $A$ has the CBAP then $\rcrs$ has the CBAP.
To finish this paper we give an example of an additional assumption under which this implication can be recovered for weakly amenable groups.

\begin{proposition}\label{pr:weakamenabilitycovarianceequivalence}
Let $(A,G,\alpha)$ be a $C^*$-dynamical system with $G$ a discrete group. 
The following are equivalent:
\begin{enumerate}[i.]
    \item $G$ is weakly amenable, $A$ has the CBAP and the approximating maps $\phi_i : A \to A$ satisfy $\phi_i \circ \alpha_t = \alpha_t \circ \phi_i$ for all $t \in G$;
    \item $(A,G,\alpha)$ is weakly amenable and the approximating Herz--Schur $(A,G,\alpha)$-multipliers $F_i : G \to \CB$ satisfy $F_i(t)(\alpha_r(a)) = \alpha_r(F_i(t)(a))$ for all $r,t \in G$. 
\end{enumerate}
\end{proposition}
\begin{proof}
(i)$\implies$(ii) The condition on the maps $(\phi_i)$ implies that the map
\[
    \tilde{\phi_i} : \rcrs \to \rcrs ;\ \sum_t \pi(a_t)\lambda_t \mapsto \sum_t \pi \big( \phi_i(a_t) \big)\lambda_t , \quad a_t \in A,\ t \in G ,
\]
can be identified with the restriction of $I_{\ell^2(G)} \otimes \phi_i$ on $\Bd[\ell^2(G)] \otimes_\mathrm{min} A$ to \rcrs. It follows from \cite[Lemma 1.5]{dCH85} that $\tilde{\phi_i}$ is completely bounded and $\nrm[\tilde{\phi_i}]_\text{cb} \leq \nrm[\phi_i]_\text{cb}$.
Let $(v_\gamma)$ be a net of scalar-valued functions on $G$ satisfying weak amenability of $G$ and let $S_{v_\gamma}$ be the completely bounded map on \rcrs\ associated to the (classical) Herz--Schur multiplier $v_\gamma$ as in \cite[Proposition 4.1]{MTT18}. 
Denote by $S_{\gamma,i}$ the composition $S_{v_\gamma} \circ \tilde{\phi_i}$, which implement the CBAP for \rcrs; indeed if $\sup_i \nrm[\phi_i]_\text{cb} \leq C_1$ and $\sup_\gamma \nrm[v_\gamma]_\Mcb \leq C_2$ then $\sup \nrm[S_{\gamma,i}]_\text{cb} \leq C_1 C_2$, each $S_{\gamma,i}$ is finite rank, and for any $T \in \rcrs$
\[
\begin{split}
    \nrm[S_{\gamma,i}(T) - T] &\leq \nrm[S_{v_\gamma} (\tilde{\phi_i}(T)) - S_{v_\gamma}(T)] + \nrm[S_{v_\gamma}(T) - T] \\
        &\leq C_2\nrm[\tilde{\phi_i}(T) - T] + \nrm[S_{v_\gamma}(T) - T] \to 0 .
\end{split}\]
It follows from Theorem~\ref{th:systemweaklyamenableiffnormclosedcrossedprodhasCBAP} that the system $(A,G,\alpha)$ is weakly amenable.
To prove the covariance condition we first calculate the form of the Herz--Schur $(A,G,\alpha)$-multipliers defined in the proof of Theorem~\ref{th:systemweaklyamenableiffnormclosedcrossedprodhasCBAP}:
\[
\begin{split}
    F_{\gamma,i}(t)(a) &:= \Big( \EE \big( S_{\gamma,i}(\pi(a)\lambda_t)\lambda_t^* \big) \Big) \\
        &= \EE \Big( \pi \big( v_\gamma(t)\phi_i(a) \big) \Big) \\
        &= v_\gamma(t) \phi_i(a) .
\end{split}
\]
Thus, for any $r \in G$,
\[
    \alpha_r \big( F_{\gamma,i}(t)(a) \big) = v_\gamma(t)\alpha_r \big( \phi_i(a) \big) = v_\gamma(t) \phi_i \big(\alpha_r(a) \big) = F_{\gamma,i}(t) \big( \alpha_r(a) \big) .
\]

(ii)$\implies$(i) Let $(F_i)$ be a net of Herz--Schur $(A,G,\alpha)$-multipliers satisfying weak amenability of the system and the covariance condition. 
Weak amenability of $G$ follows as in Remark~\ref{re:weakamenabilityofsystemimpliesofgroupandvNsystems}.
Define
\[
    \phi_i : A \to A ;\ a \mapsto \EE \Big( S_{F_i} \big( \pi(a) \big) \Big) , \quad a \in A ,
\]
to obtain a net of maps easily seen to satisfy the CBAP for $A$. 
Now calculate
\[
\begin{split}
    \phi_i \big(\alpha_t(a) \big) &= \EE \Big( S_{F_i} \big( \pi(\alpha_t(a)) \big) \Big)
        = \EE \Big( \pi \big( F_i(e)\big( \alpha_t(a) \big) \big) \Big) 
        = \EE \Big( \pi \big( \alpha_t \big( F_i(e)(a) \big) \big) \Big) \\
        &= \alpha_t \big( F_i(e)(a) \big) 
        = \alpha_t \Big( \EE \big( S_{F_i}(\pi(a)) \big) \Big) 
        = \alpha_t \big( \phi_i(a) \big) ,
\end{split}
\]
as required.
\end{proof}

\noindent
{\bf Acknowledgements.} My sincere thanks to my advisor Ivan Todorov for his guidance during this work. I would also like to thank the EPSRC for funding my PhD position.

\bibliography{weakamenabilitybib}

\end{document}